\documentclass[a4paper,10pt,twoside]{amsart}
\usepackage{amsfonts, amsmath, amssymb, appendix, mathrsfs, esint}
\usepackage{enumerate}
\usepackage{enumitem}
\usepackage{graphicx}
\numberwithin{equation}{section}
\usepackage{hyperref}

\makeatletter
\DeclareRobustCommand*{\bfseries}{%
  \not@math@alphabet\bfseries\mathbf
  \fontseries\bfdefault\selectfont
  \boldmath
}
\makeatother

\newcommand{\overseq}[1]{\overline{#1}^{\rm {\it s}} } 
\newcommand{\overseqw}[1]{\overline{#1}^{\rm {\it sw}} } 
\def\wto{\rightharpoonup}

\def\RR{\mathbb{R}}
\def\rk{{\rm rk}}
\def\tr{{\rm tr}}

\def\MM{\mathbb{M}}
\def\NN{\mathbb{N}}

\def\V{\mathcal V}

\DeclareMathOperator*{\ssetminus}{\!\setminus\!}

\def\ssup{\!>\!}
\def\sinf{\!<\!}

\def\L{J}

\def\dom{{\rm dom}}

\def\V{\mathcal V}

\def\inte{{\rm int}}

\def\eps{\varepsilon}

\def\D{D}

\DeclareMathAlphabet{\mathpzc}{OT1}{pzc}{m}{it}

\renewcommand{\limsup}{\varlimsup}
\renewcommand{\liminf}{\varliminf}

\newtheorem{theorem}{Theorem}[section]
\newtheorem{lemma}{Lemma}[section]

\newtheorem{proposition}{Proposition}[section]
\newtheorem{corollary}{Corollary}[section]
\newtheorem{definition}{Definition}[section]

\theoremstyle{example}
\newtheorem{example}{Example}[section]
\theoremstyle{remark}
\newtheorem{remark}{Remark}[section]

 \numberwithin{equation}{section}

 \newlist{hypA}{enumerate}{1}
\setlist[hypA,1]{label={\rm (${\rm A}_{\arabic*}$)}}

\newlist{hypB}{enumerate}{1}
\setlist[hypB,1]{label={\rm (${\rm B}_{\arabic*}$)}}

\newlist{hypC}{enumerate}{1}
\setlist[hypC,1]{label={\rm (${\rm C}_{\arabic*}$)}}

\newlist{hypD}{enumerate}{1}
\setlist[hypD,1]{label={\rm (${\rm D}_{\arabic*}$)}}

\newlist{hypdef}{enumerate}{1}
\setlist[hypdef,1]{label={\rm (${\rm Def}_{\arabic*}$)}}

\newlist{hyp1}{enumerate}{1}
\setlist[hyp1,1]{label={\rm ({\roman*})}}

\newlist{hypH}{enumerate}{1}
\setlist[hypH,1]{label={\rm (${\rm H}_{\arabic*}^\prime$)}}

\newlist{hypHH}{enumerate}{1}
\setlist[hypHH,1]{label={\rm (${\rm H}_{\arabic*}$)}}

\title[Radial representation of lower semicontinuous envelope]{Radial representation of lower semicontinuous envelope}

\author[Omar Anza Hafsa]{Omar Anza Hafsa}
\author[Jean-Philippe Mandallena]{Jean-Philippe Mandallena}
\address{Laboratoire LMGC, UMR-CNRS 5508, Place Eug\`ene Bataillon, 34095 Montpellier, France.}
\address{UNIVERSITE de NIMES, Laboratoire MIPA, Site des Carmes, Place Gabriel P\'eri, 30021 N\^\i mes, France}
\email{Omar Anza Hafsa <omar.anza-hafsa@unimes.fr>}
\email{Jean-Philippe Mandallena <jean-philippe.mandallena@unimes.fr>}

\begin{document}
\begin{abstract} We give an extension to a nonconvex setting of the classical radial representation result for lower semicontinuous envelope of a convex function on the boundary of its effective domain. We introduce the concept of radial uniform upper semicontinuity which plays the role of convexity, and allows to prove a radial representation result for nonconvex functions. An application to the relaxation of multiple integrals with constraints on the gradient  is given.
\end{abstract}
\maketitle
\section{Introduction} In convex analysis, a convex and lower semicontinuous function is represented on the boundary of its effective domain by a radial limit along segment from an interior point. Our goal is to give an extension to a nonconvex setting of this result. More precisely, if $X$ is a topological vector space, the problem consists to find general conditions on $f: X\to ]-\infty,\infty]$ and $D\subset\dom f$ such that for some $u_0\in D$ the following implication holds 
\begin{align*}
\forall u\!\in \!{D}\;\;\;\;\; \liminf_{D\ni v\to u}\!\!f(v)=f(u)\implies \forall u\!\in\! \partial{D}\;\;\;\;\; \liminf_{D\ni v\to u}\!\!f(v)=\lim_{t\uparrow 1}f(tu+(1-t)u_0).
\end{align*}

Our motivation comes from the theory of relaxation in the calculus of variations with constraints which consists to the study of the integral representation of the lower semicontinuous envelope of integral functionals subjected to constraints on the gradient. For convex constraints and when the lower semicontinuous envelope is convex, the radial representation on the boundary holds and allows, under some additional requirements, to extend the integral representation to the whole effective domain of the functional (see for instance~\cite{carbone-dearcangelis02}). However, even for convex constraints, the lower semicontinuous envelope is not necessarily convex when the gradient is a matrix. Indeed, in this case it is well known that if an integral representation holds then, usually, the ``relaxed" integrand is quasiconvex or rank-one convex (in the sense of Morrey, see for instance~\cite{dacorogna08}). Therefore, the need of a generalization of the radial representation to a nonconvex setting comes naturally. 

The analysis of how the convexity concept plays to obtain the radial limit representation highlights some kind of uniform upper semicontinuity property, more precisely, when $f:X\to]-\infty,\infty]$ is convex we may write 
\begin{align}\label{eqintro}
\sup_{u\in\dom f}\frac{f(tu+(1-t)u_0)-f(u)}{1+\lvert f(u_0)\rvert+\lvert f(u)\rvert}\le 1-t
\end{align}
for any $u_0\in\dom f$. The left hand side of~\eqref{eqintro} is a kind of uniform semicontinuity modulus which is lower than $0$ when $t\uparrow 1$. This is exactly the property we need to overcome the convex case. For a non necessary convex $f$ satisfying~\eqref{eqintro}, we will say that $f$ is radially uniformly upper semicontinuity (see Definition~\ref{ru-usc-Def-functional}). 

The concept of radial uniform upper semicontinuity, to our best of knowledge, finds its origin in~\cite[Condition (10.1.13), p. 213]{carbone-dearcangelis02} in connection with relaxation problems with constraints. Later, this concept was proved very useful for relaxation problems in the vectorial case with bounded and convex constraints see~\cite{oah10}. Then, it was used to study several homogenization and relaxation problems with constraints (see for instance~\cite{oah-jpm11,oah-jpm12, jpm11arxiv, oah-jpm12arxiv}). 
 
 \medskip

The plan of the paper is as follows. In Section~\ref{defprel}, we first give the definitions of radially uniformly upper semicontinuous functions and star-shaped sets (stated here in an infinite dimensional framework) which play the role of convexity concepts in a nonconvex setting. 

In Section~\ref{mresult}, we state and prove the main result Theorem~\ref{GENERAL-FORMULA}, which gives a radial representation of a lower semicontinuous envelope on the boundary of a star-shaped set. A sequential version of Theorem~\ref{GENERAL-FORMULA} is stated and proved also.

In Section~\ref{consgene}, we deduce some general consequences from Theorem~\ref{GENERAL-FORMULA}. In particular, we deal with the case of convex functions and minimization problems with star-shaped constraints. 

In Section~\ref{2}, we study the stability of radially uniformly upper semicontinuity concept with respect to calculus operations. We also give examples of some general class of radially uniformly upper semicontinuous functions.

Section~\ref{4} is devoted to an application of a relaxation problem of the calculus of variations with constraints on the gradient.

\section{Definitions and preliminaries}\label{defprel}
Let $X$ be an Hausdorff topological vector space. For a function $f:X\to ]-\infty,\infty]$ we denote its effective domain by
\begin{align*}
\dom f:=\{u\in X:f(u)\sinf\infty\}.
\end{align*}
For each $a\ssup 0$, $D\subset \dom f$ and $u_0\in D$, we define $\Delta_{f,D,u_0}^a:[0,1]\to ]-\infty,\infty]$ by
\begin{align*}
\Delta_{f,D,u_0}^a(t):=\sup_{u\in D}\frac{f(tu+(1-t)u_0)-f(u)}{a+\vert f(u)\vert}.
\end{align*}
If $D=\dom f$ then we write $\Delta_{f,u_0}^a:=\Delta_{f,D,u_0}^a$.
\begin{remark} By analogy with the case of uniform continuity, the function $\Delta_{f,D,u_0}^a(\cdot)$ can be seen as a {\em (semi) radial uniform modulus of $f$ on $D$ relatively to $u_0$}. 
\end{remark}

\begin{definition}\label{ru-usc-Def-functional}$ $
\begin{enumerate}
\item[{\rm (1)}]\label{def1} Let $D\subset \dom f$ and $u_0\in D$. We say that $f$ is {\em radially uniformly upper semicontinuous in $D$ relative to $u_0$}, if there exists $a\ssup 0$ such that
\begin{align*}
 \limsup_{t\to 1}\Delta^a_{f,D,u_0}(t)\leq 0.
\end{align*}
Radially uniformly upper semicontinuous will be abbreviated to {\em ru-usc} in what follows.

If $D=\dom f$ then we simply say that $f$ is {\em ru-usc} relative to $u_0\in \dom f$.
\item[{\rm (2)}]\label{def2} We say that $D\subset X$ is a {\em strongly star-shaped set relative to $u_0\in D$}, if 
\begin{align*}
t\overline{D}+(1-t)u_0\subset D\;\hbox{ for all }t\in[0,1[,
\end{align*}
where $\overline D$ is the closure of $D$ in $X$.
\end{enumerate}
\end{definition}

When $D\subset \dom f$ is strongly star-shaped  relative to $u_0\in D$ and $f$ is ru-usc in $D$ relative to $u_0\in D$, we say that {\em $f$ is ru-usc in the strongly star-shaped set $D$ relative to $u_0\in D$}.

\begin{remark}\label{convex stars1} $ $
\begin{hyp1}
\item Our definition of strongly star-shaped sets is more restrictive than the usual one (see~\cite{valentine}) which requires 
\begin{align*}
t{D}+(1-t)u_0\subset D\;\hbox{ for all }t\in[0,1].
\end{align*}
\item\label{convex stars} When $X$ is a normed vector space, a convex set $D\subset X$ with nonempty interior is strongly star-shaped relative to any $u_0\in \inte(D)$, this fact is also known as the {\em line segment principle} (see~\cite[Theorem 2.33, p. 58]{rockafellar-wets98}). Moreover if $D\subset X$ is strongly star-shaped relative to all $u_0\in D$ then $D$ is convex. An example of strongly star-shaped set which is not necessarily convex is given by the union of two different convex sets $D_1\cup D_2$ with $\inte(D_1\cap D_2)\not=\emptyset$, indeed if $u_0\in\inte(D_1\cap D_2)$ then for each $u\in \overline{D_1\cup D_2}$ either $u\in \overline{D_1}$ or $u\in \overline{D_2}$, so $tu+(1-t)u_0\in D_1$ or $tu+(1-t)u_0\in D_2$ for all $t\in[0,1[$ since $D_1$ and $D_2$ are convex with nonempty interior and then strongly star-shaped relative to $u_0$.
\end{hyp1}
\end{remark}

\begin{proposition}\label{convexruusc} Let $f:X\to ]-\infty,\infty]$ be a convex function. Let $D\subset \dom f$ and $u_0\in D$. Then $f$ is ru-usc in $D\subset\dom f$ relative to $u_0\in D$.
\end{proposition}
\begin{proof} By convexity of $f$ it is easy to see that for all $t\in [0,1[$
\begin{align*}
\Delta^{1+\lvert f(u_0)\rvert}_{f,D,u_0}(t)=\sup_{u\in D}\frac{f(tu+(1-t)u_0)-f(u)}{1+\lvert f(u_0)\rvert+\lvert f(u)\rvert}\le 1-t
\end{align*}
which shows, by letting $t\to 1$, that $f$ is ru-usc in $D$ relative to $u_0$. 
\end{proof}

For each set $D\subset X$ we denote by $\chi_D$ the indicator function of $D$ given by
\begin{align*}\chi_D(u):=
\left\{
\begin{array}{cl}
  0 &\mbox{ if }u\in D      \\
 \infty &  \mbox{ if }u\in X\setminus D.
\end{array}
\right.
\end{align*}

\begin{remark}\label{indicator-ruusc} If a set $D\subset X$ is strongly star-shaped relative to $u_0\in D$ then $\chi_D$ is ru-usc relative to $u_0$. 
\end{remark}

We denote by $\overline{f}$ the lower semicontinuous envelope of $f$ given by
\begin{align*}
\overline{f}(u)=\liminf_{v\to u}f(v):=\sup_{U\in \V(u)}\inf_{v\in U}f(v)
\end{align*}
where $\V(u)$ denotes the set of neighborhoods of $u\in X$.

The convergence of a sequence $\{u_n\}_n\subset X$ to $u$ with respect to the topology of $X$ is denoted by $u_n{\to}u$. We denote by $\overseq{f}$ the {\em sequential relaxation} of $f$ given by
\begin{align*}
\overseq{f}(u)=\inf\left\{\liminf_{n\to \infty}f(u_n):X\ni u_n{\to} u\right\}.
\end{align*}

\begin{remark}\label{envelope}
It is easy to see that $\overline{f}\le f$ and $\overline{\overline{f}}=\overline{f}$. It is worth to note that $\overline{f}$ is sequentially lower semicontinuous, i.e., for every $\{u_n\}_n\subset X$ and $u\in X$ if $u_n\to u$ then $\liminf_{n\to \infty}\overline{f}(u_n)\ge \overline{f}(u)$. We have $\overline{f}\le \overseq{f}\le f$ and in general $\overseq{\overseq{f}}$ is different from $\overseq{f}$.
\end{remark}
For each $u_0\in\dom f$ the {\em radial extension $\widehat{f}_{u_0}:X\to [-\infty,\infty]$ relative to $u_0$} is defined by
\begin{align*}
\widehat{f}_{u_0}(u):=\liminf_{t\uparrow 1}f(tu+(1-t)u_0)
\end{align*}
where 
\begin{align*}
\liminf_{t\uparrow 1}f(tu+(1-t)u_0)=\inf\left\{\liminf_{n\to \infty} f(t_nu+(1-t_n)u_0):[0,1[\ni t_n\to 1\right\}.
\end{align*}
\begin{remark}\label{domfhat} The effective domain of $\widehat{f}_{u_0}$ satisfies
$
\dom \widehat{f}_{u_0}\subset \dom \overseq{f}\subset \dom \overline{f},
$
indeed, if we consider $u\in\dom \widehat{f}_{u_0}$ then for some $\{t_n\}_n\in [0,1[$ such that $\lim_{n\to\infty} t_n=1$, we have by Remark~\ref{envelope}
\begin{align*}
\infty\ssup\widehat{f}_{u_0}(u)=\liminf_{t\uparrow 1}f(tu+(1-t)u_0)=\lim_{n\to \infty}f(t_nu+(1-t_n)u_0)\ge \overseq f(u)\ge \overline f(u).
\end{align*}
since $t_nu+(1-t_n)u_0\to u$ as $n\to \infty$.
\end{remark}
\section{Main results}\label{mresult}
\subsection{Radial representation on the boundary of a strongly star-shaped set}
Here is the main result of the paper which establishes a radial representation of the lower semicontinuous envelope of a ru-usc function penalized by a strongly star-shaped subset.
\begin{theorem}\label{GENERAL-FORMULA} Let $f:X\to ]-\infty,\infty]$ be a ru-usc function in a strongly star-shaped set $D\subset \dom f$ relative to $u_0\in D$. Assume that $\inf_D f\ssup-\infty$, we have 
\begin{hyp1}
\item\label{gene1} if $\displaystyle \overline{f+\chi_D}=f$ on $D$ then $\displaystyle\overline{f+\chi_D}=\widehat{f}_{u_0}+\chi_{\overline{D}}$;\\
\item\label{gene2} $\displaystyle\widehat{f}_{u_0}(u)=\lim_{t\uparrow 1}f(tu+(1-t)u_0)$ for all $u\in \overline{D}$, and $\displaystyle\widehat{f}_{u_0}$ is ru-usc in $\overline{D}\cap\dom \widehat{f}_{u_0}$.
\end{hyp1} 
\end{theorem}
\begin{remark} The condition in Theorem~\ref{GENERAL-FORMULA}~\ref{gene1}, i.e. $\displaystyle \overline{f+\chi_D}=f$ on $D$, means that {\em $f$ is lower semicontinuous in $D$} in the sense that $\liminf_{D\ni v\to u}f(v)= f(u)$ for all $u\in D$. Thus Theorem~\ref{GENERAL-FORMULA}~\ref{gene1} can be rewritten as 
\begin{align*}
\forall u\in {D}\; \liminf_{D\ni v\to u}f(v)=f(u)\implies \forall u\in \partial{D}\; \liminf_{D\ni v\to u}f(v)=\liminf_{t\uparrow 1}f(tu+(1-t)u_0).
\end{align*}
\end{remark}
\noindent{\em Proof of Theorem~\ref{GENERAL-FORMULA}}
\noindent{\em Proof of~\ref{gene1}.} Fix $u\in\overline D$. We have $(1-t)u_0+tu\in D$ for all $t\in [0,1[$ since $D$ is strongly star-shaped relative to $u_0$. Hence
\begin{align*}
\widehat{f}_{u_0}(u)&=\liminf_{t\uparrow 1}f(tu+(1-t)u_0)\\
&=\liminf_{t\uparrow 1}f(tu+(1-t)u_0)+\chi_D(tu+(1-t)u_0)\\
&\ge \overline{f+\chi_D}(u).
\end{align*}

It remains to prove that $\widehat{f}_{u_0}+\chi_{\overline D}\le \overline{f+\chi_D}$. It is equivalent to show that for every $\delta\ssup 0$ and every $u\in\overline D$ we have $\widehat{f}_{u_0}(u)\sinf\delta$ whenever $\overline{f+\chi_D}(u)\sinf\delta$.

Fix $\delta\ssup0$ and $u\in \overline D$ such that $\overline{f+\chi_D}(u)\sinf\delta$. Then there exists $\eta\ssup 0$ such that for every $U\in\V(u)$ it holds
\begin{align}\label{gf2}
\inf_{v\in U}\left(f+\chi_D\right)(v)=\inf_{v\in U\cap D} f(v)\le \delta-2\eta.
\end{align}
For each $U\in\V(u)$ there exists $v_U\in U\cap D$ (Note that $U\cap D\not=\emptyset$ for any $U\in\V(u)$ since $u\in\overline D$.) such that
\begin{align}\label{gf3}
f(v_U)\sinf \inf_{v\in U\cap D}f(v)+\eta\le \delta-\eta
\end{align}
since~\eqref{gf2}. 

Let $\{t_n\}_n\in[0,1[$ be such that $\lim_{n\to\infty} t_n=1$, and for every $n\in\NN$
\begin{align}
\Delta_{f,D,u_0}^a(t_{n})\le \frac{1}{n+1}\label{gf4.3}
\end{align}
since $f$ is ru-usc in $D$ relative to $u_0$.

Fix $U\in\V(u)$. By~\eqref{gf3} and~\eqref{gf4.3} we have 
\begin{align}\label{gf5}
f(t_{n} v_U+(1-t_{n})u_0)&\le \Delta_{f,D,u_0}^a(t_{n})\left(a+\lvert f(v_U)\rvert\right)+f(v_U)\\
&\le \frac{1}{n+1}\left(a+\max\left\{-\inf_D f,\delta-\eta\right\}\right)+\delta-\eta.\notag
\end{align}
On the other hand, for every $n\in\NN$ it holds
\begin{align*}
f(t_{n} v_U+(1-t_{n})u_0)\ge \inf_{v\in U\cap D} f(t_{n} v+(1-t_{n})u_0).
\end{align*}
We have $t_n u+(1-t_n)u_0\in D$ for all $n\in\NN$ since $D$ is strongly star-shaped relative to $u_0$. So, by using~\eqref{gf5}, the assumption $\overline{f+\chi_D}=f$ on $D$ we obtain for every $n\in\NN$
\begin{align*}
 \delta-\eta&\ge\sup_{U\in \V(u)}\inf_{v\in U\cap D} f(t_{n} v+(1-t_{n})u_0)\\
&=\overline{f+\chi_D}(t_{n} u+(1-t_{n})u_0)\\
&=f(t_{n} u+(1-t_{n})u_0).
\end{align*}
Letting $n\to\infty$ we find that $\widehat{f}_{u_0}(u)\le\delta-\eta$, which completes the Proof of~\ref{gene1}.

\noindent{\em Proof of~\ref{gene2}.} We first have to prove that for every $u\in \overline{D}$
\begin{align*}
\liminf_{t\to 1^-}f(tu+(1-t)u_0)=\limsup_{t\to 1^-}f(tu+(1-t)u_0).
\end{align*}

Fix $u\in \overline{D}$. It suffices to prove that
\begin{equation}\label{ru-usc-AssUmPtIoN1}
\limsup_{t\to 1}\Psi(t)\leq\liminf_{t\to 1}\Psi(t).
\end{equation}
where $\Psi(t):=f(tu+(1-t)u_0)$ for all $t\in [0,1]$. Without loss of generality we can assume that $\liminf_{t\to 1}\Psi(t)\sinf\infty$. Choose two sequences $\{t_n\}_n, \{s_n\}_n\subset]0,1[$ such that\footnote{Once the sequences $\{t_n\}_n, \{s_n\}_n\subset]0,1[$ satisfying $t_n\to 1$, $s_n\to 1$ choosen, we can extract a subsequence $\{s_{\sigma(n)}\}_n$ such that $\frac{t_n}{s_{\sigma(n)}}< 1$ for all $n\in\NN$. Indeed, it suffices to consider the increasing map $\sigma:\NN\to\NN$ defined by $\sigma(0):=\min\{\nu\in\NN:s_\nu>t_0\}$ and $\sigma(n+1):=\min\{\nu\in\NN:\nu>\sigma(n)\mbox{ and }s_\nu>t_{n+1}\}$.} $t_n\to 1$, $s_n\to 1$, $\frac{t_n}{s_n}< 1$ for all $n\in\NN$, and 
\begin{align*}
\limsup_{t\to 1}\Psi(t)&=\lim_{n\to\infty}\Psi(t_n);\\
\liminf_{t\to 1}\Psi(t)&=\lim_{n\to\infty}\Psi(s_n).
\end{align*}
Since $D$ is strongly star-shaped relative to $u_0$ we have $t_nu+(1-t_n)u_0\in D$ for all $n\in\NN$, so we can assert that for every $n\in\NN$
\begin{align}\label{ru-usc-AssUmPtIoN2}
\Psi(t_n)&=\Psi\left(\frac{t_n}{s_n}s_n\right)\\
&=f\left(\frac{t_n}{s_n} (s_nu+(1-s_n)u_0)+\left(1-\frac{t_n}{s_n}\right)u_0\right)\notag\\
&\leq \Delta^a_{f,D,u_0}\left(\frac{t_n}{s_n}\right)(a+\lvert \Psi(s_n)\rvert)+\Psi(s_n).\notag
\end{align}
We have that~\eqref{ru-usc-AssUmPtIoN1} follows from~\eqref{ru-usc-AssUmPtIoN2} by letting $n\to\infty$ since $f$ is ru-usc in $D$ relative to $u_0$.

It remains to prove that $\displaystyle\widehat{f}_{u_0}$ is ru-usc in $\overline{D}\cap\dom \widehat{f}_{u_0}$. Fix $t\in[0,1[$ and $u\in\overline{D}\cap\dom \widehat{f}_{u_0}$. By the first part of the Proof of~\ref{gene2} we can assert that
\begin{align*}
\widehat{f}_{u_0}(u)&=\lim_{s\to 1}f(su+(1-s)u_0)\\
\widehat{f}_{u_0}(tu+(1-t)u_0)&=\lim_{s\to 1}f(s(tu+(1-t)u_0)+(1-s)u_0).
\end{align*}
So, we have
\begin{align*}
\begin{split}
&\frac{\widehat{f}_{u_0}(tu+(1-t)u_0)-\widehat{f}_{u_0}(u)}{a+\vert\widehat{f}_{u_0}(u)\vert}\\
&\phantom{=}=\lim_{s\to 1}\frac{f(t(su+(1-s)u_0)+(1-t)u_0)-f(su+(1-s)u_0)}{a+\vert f(su+(1-s)u_0)\vert}\\
&\phantom{=}\le \Delta_{f,D,u_0}^a(t).
\end{split}
\end{align*}
It follows that $\Delta_{f,\overline{D}\cap\dom \widehat{f}_{u_0},u_0}^a(t)\le  \Delta_{f,D,u_0}^a(t)$ and the proof is complete by letting $t\to 1$. 
\hfill$\blacksquare$
\subsection{Sequential version of Theorem~\ref{GENERAL-FORMULA}} It is well known that if $X$ has a countable base of neighborhoods of $0$ then $\overseq{f}=\overline{f}$. For a subset $D\subset X$, we denote by $\overseq{D}$ the sequential closure of $D$, i.e., $u\in \overseq{D}$ if and only if there exists a sequence $\{v_n\}_n\subset D$ such that $v_n\to u$ as $n\to\infty$. We say that $D\subset X$ is a {\em sequentially strongly star-shaped set relative to $u_0\in D$}, if 
\begin{align*}
t\overseq{D}+(1-t)u_0\subset D\;\hbox{ for all }t\in[0,1[.
\end{align*}

Here is a sequential version of Theorem~\ref{GENERAL-FORMULA}. 
\begin{theorem}\label{GENERAL-FORMULA-SEQ} Let $f:X\to ]-\infty,\infty]$ be a ru-usc function in a strongly star-shaped set $D\subset \dom f$ relative to $u_0\in D$. Assume that $\inf_D f\ssup -\infty$, we have 
\begin{hyp1}
\item\label{gene1seq} if $\displaystyle \overseq{f+\chi_D}=f$ on $D$ then $\displaystyle\overseq{f+\chi_D}=\widehat{f}_{u_0}+\chi_{\overseq{D}}$;\\
\item\label{gene2seq} $\displaystyle\widehat{f}_{u_0}(u)=\lim_{t\uparrow 1}f(tu+(1-t)u_0)$ for all $u\in \overseq{D}$, and $\displaystyle\widehat{f}_{u_0}$ is ru-usc in $\overseq{D}\cap\dom \widehat{f}_{u_0}$.
\end{hyp1} 
\end{theorem}
\begin{proof}For the sake of completeness we give the proof of~\ref{gene1seq} although very similar to the one of Theorem~\ref{GENERAL-FORMULA}~\ref{gene1}.

Fix $u\in\overseq D$. Let $\{t_n\}_n\subset [0,1[$ such that $\lim_{n\to\infty}t_n=1$ and \[\widehat{f}_{u_0}(u)=\lim_{n\to\infty} f(t_nu+(1-t_n)u_0).\] We have $(1-t_n)u_0+t_n u\in D$ for all $n\in\NN$ since $D$ is strongly star-shaped relative to $u_0$. Hence
\begin{align*}
\widehat{f}_{u_0}(u)&=\lim_{n\to\infty} f(t_nu+(1-t_n)u_0)\\
&=\lim_{n\to\infty} f(t_nu+(1-t_n)u_0)+\chi_D(t_nu+(1-t_n)u_0)\\
&\ge \overseq{f+\chi_D}(u)
\end{align*}
since $t_nu+(1-t_n)u_0\to u$ as $n\to \infty$.

It remains to prove that $\widehat{f}_{u_0}+\chi_{\overseq{D}}\le \overseq{f+\chi_D}$. Fix $\delta\ssup0$ and $u\in \overseq D$ such that $\overseq{f+\chi_D}(u)\sinf\delta$. So we can find $\eta\in ]0,\delta[$ and a sequence $\{v_n\}_n\subset D$ such that 
\begin{align*}
v_n\to u\quad\mbox{ and }\quad\lim_{n\to\infty}f(v_n)\le\delta-\eta.
\end{align*}
Choose a subsequence $\{v_n\}_n\subset D$ (not relabelled) such that 
 \begin{align}\label{gf3seq}
v_n\to u\quad\mbox{ and }\quad \lvert f(v_n)\rvert \le\max\left\{-\inf_D f,\delta-\eta\right\} \mbox{ for all }n\in\NN.
\end{align}
Let $\{t_k\}_k\subset [0,1[$ such that $\lim_{k\to\infty}t_k=1$, and for every $k\in\NN$
\begin{align}
&\Delta_{f,D,u_0}^a(t_{k})\le \frac{1}{k+1}\label{gf4.3seq}
\end{align}
since $f$ is ru-usc in $D$ relative to $u_0$.
So, by~\eqref{gf3seq} and~\eqref{gf4.3seq}, for every $k,n\in\NN$ we have 
\begin{align}\label{gf5seq}
f(t_{k} v_n+(1-t_{k})u_0)&\le \Delta_{f,D,u_0}^a(t_{k})\left(a+\lvert f(v_n)\rvert\right)+f(v_n)\\
&\le \frac{1}{k+1}\left(a+\max\left\{-\inf_D f,\delta-\eta\right\}\right)+\delta-\eta.\notag
\end{align}
We have $t_k u+(1-t_k)u_0\in D$ for all $k\in\NN$ since $D$ is strongly star-shaped relative to $u_0$. Letting $n\to\infty$ in~\eqref{gf5seq} and using the assumption $\overline{f+\chi_D}=f$ on $D$ we obtain for every $k\in\NN$
\begin{align*}
f(t_{k} u+(1-t_{k})u_0)&=\overline{f+\chi_D}(t_{k} u+(1-t_{k})u_0)\\
&\le\frac{1}{k+1}\left(a+\max\left\{-\inf_D f,\delta-\eta\right\}\right)+\delta-\eta
\end{align*} 
Letting $k\to\infty$ we find that $\widehat{f}_{u_0}(u)\le\delta-\eta$, which completes the Proof of~\ref{gene1seq}.

The proof of~\ref{gene2seq} is the same as the proof of Theorem~\ref{GENERAL-FORMULA}~\ref{gene2}. 
\end{proof}
\section{General consequences}\label{consgene}
\subsection{Radial representation with the effective domain as constraint}
When $D=\dom f$ we have the following result.

\begin{corollary}\label{GENERAL-FORMULA2} Let $f:X\to ]-\infty,\infty]$ be a function such that $\overline f$ is ru-usc in the strongly star-shaped set $\dom f$ relative to $u_0\in \dom f$. If $g: X\to ]-\infty,\infty]$ is such that $\overline{f}=g$ on $\dom f$ and $\inf_{\dom f}g\ssup-\infty$ then
\begin{align*}
\overline{f}=\widehat{g}_{u_0}+\chi_{\dom\overline{f}}
\end{align*}
where $\widehat{g}_{u_0}(u)=\lim_{t\uparrow 1}g(tu+(1-t)u_0)$ for all $u\in \overline{\dom f}$.
\end{corollary}
\begin{proof} Since $\overline{f}=g$ on $\dom f$ and $\dom f$ is strongly star-shaped relative to $u_0\in \dom f$ the function $g$ is ru-usc in $\dom f$ relative to $u_0$. Thus applying Theorem~\ref{GENERAL-FORMULA}~\ref{gene2} we obtain for every $u\in\overline{\dom f}$
\begin{align*}
\widehat{g}_{u_0}(u)=\lim_{t\to1^-}g(tu+(1-t)u_0).
\end{align*}
Now, we have to prove that $\overline{f}=\widehat{g}_{u_0}+\chi_{{\dom \overline{f}}}$. It is easy to see that
\[
\overline{f}\le \overline{f}+\chi_{\dom f}\le f+\chi_{\dom f} 
\]
hence $\overline{\overline{f}+\chi_{\dom f}}=\overline{f}$ since $ \overline{f+\chi_{\dom f}}=\overline{f}$ and $\overline{\overline{f}}=\overline{f}$. Apply Theorem~\ref{GENERAL-FORMULA}~\ref{gene1} with $\overline{f}$ in place of $f$ and $D=\dom f$, we obtain 
$
\overline{\overline{f}+\chi_{\dom f}}=\overline{f}=\widehat{\overline{f}}_{u_0}+\chi_{\overline{\dom f}}.
$
Taking account of Remark~\ref{domfhat} we deduce that 
\[
\overline{f}=\widehat{\overline{f}}_{u_0}+\chi_{\dom \overline{f}}.
\]
It remains to prove that $\widehat{\overline{f}}_{u_0}=\widehat{g}_{u_0}$ on ${\dom \overline{f}}$. Fix $u\in {\dom \overline{f}}$, then $tu+(1-t)u_0\in\dom f$ for all $t\in [0,1[$ since $\dom f$ is strongly star-shaped relative to $u_0$. It follows that
\begin{align*}
\widehat{\overline{f}}_{u_0}(u)=\liminf_{t\uparrow 1}\overline{f}(tu+(1-t)u_0)=\liminf_{t\uparrow 1}g(tu+(1-t)u_0)=\widehat{g}_{u_0}(u)
\end{align*}
since $\overline{f}=g$ on $\dom f$, which completes the proof.
\end{proof}
\begin{remark} The previous result can be useful in relaxation problems, indeed, in practice we are able to prove an integral representation of $\overline f$, say $g$, on $\dom f$ only. Then we can use Corollary~\ref{GENERAL-FORMULA2} to have a representation on $\dom \overline{f}$. To obtain a full integral representation on $\dom \overline{f}$, we have then to commute ``$\lim_{t\uparrow 1}$" with the integration in the radial limit $\widehat{g}_{u_0}$ (see for instance Theorem\ref{quasiconstraints}).
\end{remark}
Analysis similar to that in the proof of Corollary~\ref{GENERAL-FORMULA2} gives the following sequential version.
\begin{corollary}\label{GENERAL-FORMULA2seq} Let $f:X\to ]-\infty,\infty]$ be a function such that $\overseq f$ is ru-usc in the strongly star-shaped set $\dom f$ relative to $u_0\in \dom f$. If $g: X\to ]-\infty,\infty]$ is such that $\overseq{f}=\overseq{g}=g$ on $\dom f$ and $\inf_{\dom f}g\ssup-\infty$ then
\begin{align*}
\overseq{f}=\widehat{g}_{u_0}+\chi_{\dom\overseq{f}}
\end{align*}
where $\widehat{g}_{u_0}(u)=\lim_{t\uparrow 1}g(tu+(1-t)u_0)$ for all $u\in \overseq{\dom f}$.
\end{corollary}

In the following, we state a consequence of Corollary~\ref{GENERAL-FORMULA2} and Corollary~\ref{GENERAL-FORMULA2seq} in the particular case where $g$ is replaced by $f$.

\begin{corollary}\label{ruusc-representation} Let $f:X\to ]-\infty,\infty]$ be a function such that $\overline f$ (resp. $\overseq f$) is ru-usc in the strongly star-shaped set $\dom f$ relative to $u_0\in\dom f$. If $\overline f=f$ (resp. $\overseq f=f$) on $\dom f$ and $\inf_{\dom f}f\ssup-\infty$ then 
\begin{align*}
\overline f=\widehat{f}_{u_0}\;{\rm (\mbox{resp. }}\overseq f=\widehat{f}_{u_0}{\rm )}.
\end{align*}
\end{corollary}
\begin{proof} Apply Corollary~\ref{GENERAL-FORMULA2} (resp. Corollary~\ref{GENERAL-FORMULA2seq}) with $f$ in place of $g$, we obtain $\overline f=\widehat{f}_{u_0}+\chi_{\dom \overline f}$ (resp. $\overseq f=\widehat{f}_{u_0}+\chi_{\dom \overseq f}$). To finish the proof it suffices to see that $\dom \widehat{f}_{u_0}\subset \dom \overseq f\subset\dom \overline f$ since Remark~\ref{domfhat}. 
\end{proof}

We examine the case $f$ convex and bounded below.
\begin{corollary} Let $f:X\to ]-\infty,\infty]$ be a convex function. If there exists $u_0\in X$ such that $f$ is bounded above in a neighborhood of $u_0$ and $\inf_{\inte(\dom f)}  f\ssup-\infty$ then

\begin{align}\label{convex extension 2}
\overline f=\widehat{f}_{u_0}.
\end{align}
Moreover $\widehat{f}_{u_0}=\widehat{f}_{v}$ for all $v\in\inte(\dom f)$ (where $\inte(\dom f)$ is the interior of $\dom f$).
\end{corollary}
\begin{proof} The assumption implies that $f$ is continuous on $\inte(\dom f)$ the interior of $\dom f$. So, $\overline{f+\chi_{\inte(\dom f)}}=f$ on $\inte(\dom f)$. By Proposition~\ref{convexruusc} a convex function is ru-usc in $\inte(\dom f)$ relative to all $v\in\inte(\dom f)$. It is well known that $\inte(\dom f)$ is convex, then by Remark~\ref{convex stars1}~\ref{convex stars} the set $\inte(\dom f)$ is strongly star-shaped relative to all $v\in\inte(\dom f)$. By applying Theorem~\ref{GENERAL-FORMULA} we find $\overline f=\widehat{f}_{v}+\chi_{\overline{\inte(\dom f)}}$. On the other hand it holds that $\dom\overline{f}\subset \overline{\dom f}= \overline{\inte(\dom f)}$ since $\dom f$ is convex. Therefore~\eqref{convex extension 2} holds since Remark~\ref{domfhat}.
\end{proof}
\begin{remark} In fact equality~\eqref{convex extension 2} still holds for convex functions which are not bounded below. When $X$ has finite dimension, more general results involving convexity exist, see for instance~\cite[Theorem 7.5 p. 57, Theorem 10.3 p. 85]{rockafellar70}. Indeed, in finite dimension, the relative interior of $\dom f$ is not empty whenever $\dom f\not=\emptyset$ and $f$ is continuous on it, so~\eqref{convex extension 2} holds without any assumption on $f$ unless to be convex.
\end{remark}
\subsection{Minimization with strongly star-shaped constraints}
The following result deals with minimization problems with strongly star-shaped constraints, it can be seen as a nonconvex version of~\cite[Corollary 4.41, p. 272]{fonseca-leoni07}.
\begin{corollary}\label{inf-ruusc-cor} Let $f:X\to ]-\infty,\infty]$ be a function. Let $D\subset \dom f$ be a strongly star-shaped set relative to $u_0\in D$ such that $\overline{D}\subset\dom f$. Assume that $f:X\to ]-\infty,\infty]$ is ru-usc in $\overline D$ relative to $u_0$. If $\overline{f+\chi_D}=f$ on $D$ then
\begin{align*}
\inf_{D}f=\inf_{\overline D}f.
\end{align*}
\end{corollary}
\begin{proof} Without loss of generality we can assume that $\inf_{D}f\ssup-\infty$. It is sufficient to show that $\inf_{D}f\le \inf_{\overline D}f$. By applying Theorem~\ref{GENERAL-FORMULA} we have
\begin{align}\label{inf-ruusc}
\inf_{D}f=\inf_X\overline{f+\chi_D}=\inf_{\overline D}\widehat{f}_{u_0}.
\end{align}
Now we claim that for every $u\in\overline{D}$ we have $\widehat{f}_{u_0}(u)\le f(u)$, indeed, 
\begin{align*}
\widehat{f}_{u_0}(u)=\liminf_{t\uparrow 1}f(tu+(1-t)u_0)\le \limsup_{t\uparrow 1}\Delta_{f,\overline D,u_0}^a(t)(a+\vert f(u)\vert)+f(u)\le f(u)
\end{align*}
since $f$ is ru-usc in $\overline{D}$ relative to $u_0$. We deduce from~\eqref{inf-ruusc} that 
$
\inf_{D}f\le \inf_{\overline D}f,
$
and the proof is complete.
\end{proof}
The following result can be useful in scalar problems of the calculus of variations when the lower semicontinuous envelope $\overline f$ of a nonconvex $f$ is convex.
\begin{corollary} Let $f:X\to ]-\infty,\infty]$ be such that $\overline{f}$ is convex. Let $D\subset \dom f$ be a strongly star-shaped set relative to $u_0\in D$ such that $\overline{D}\subset\dom f$. Then
\begin{align*}
\inf_{D}\overline{f}=\inf_{\overline D}\overline{f}.
\end{align*}
\end{corollary}
\begin{proof} We have
\begin{align*}
\overline{f}\le \overline{\overline f+\chi_D}\le \overline f+\chi_D.
\end{align*}
Thus $\overline{\overline f+\chi_D}=\overline{f}$ on $D$. By Proposition~\ref{convexruusc} $\overline f$ is ru-usc in $\overline{D}\subset\dom f$ relative to any $u_0\in \overline{D}$, so we can apply Corollary~\ref{inf-ruusc-cor} with $\overline f$ in place of $f$. The proof is complete.
\end{proof}

\section{Operations on ru-usc functions}\label{2} In this section we study the stability of ru-usc functions with respect to some operations. We also give some examples of class of ru-usc functions. 
\subsection{Stability of ru-usc functions with respect to pointwise sum and product}
We need the following result in the proof of Proposition~\ref{sumproduct}.
\begin{lemma}\label{translation} Let $f:X\to]-\infty,\infty]$ be a ru-usc function in $D\subset\dom f$ relative to $u_0\in D$. Then 
\begin{hyp1}
\item\label{transla} $f+c$ is ru-usc in $D$ relative to $u_0$ for all $c\in\RR$;\\
\item\label{extp} $\lambda f$ is ru-usc in $D$ relative to $u_0$ for all $\lambda\in\RR^+$.
\end{hyp1}
\end{lemma}
\begin{proof} {\em Proof of~\ref{transla}.} Fix $u\in D$ and $c\in\RR$. Fix $\eps\ssup 0$. There exists $t_\eps\in ]0,1[$ such that $\sup_{t\in ]t_\eps,1[}\Delta_{f,D,u_0}^a(t)\sinf\eps$ since $f$ is ru-usc in $D$ relative to $u_0$.

We set $f_c:=f+c$. Then for every $t\in ]t_\eps,t[$ we have
\begin{align}\label{eqc}
f_c(tu+(1-t)u_0)-f_c(u)&=f(tu+(1-t)u_0)-f(u)\\
&\le \Delta_{f,D,u_0}^a(t)(a+\vert f(u)\vert)\notag\\
&\le \eps(a+\vert f_c(u)\vert+\vert c\vert).\notag
\end{align}
It follows that $\limsup_{t\to 1}\Delta_{f_c,D,u_0}^{a+\vert c\vert}(t)\le  \eps$. The proof of~\ref{transla} is complete by letting $\eps\to 0$. 

\medskip

\noindent{\em Proof of~\ref{extp}.} Fix $u\in D$ and $\lambda\in\RR^+$. Fix $\eps\ssup 0$. There exists $t_\eps\in ]0,1[$ such that $\sup_{t\in ]t_\eps,1[}\Delta_{f,D,u_0}^a(t)\sinf\eps$ since $f$ is ru-usc in $D$ relative to $u_0$. 

We set $f_\lambda:=\lambda f$. Then for every $t\in ]t_\eps,t[$ we have
\begin{align*}
f_\lambda (tu+(1-t)u_0)-f_\lambda(u)&=\lambda\left(f(tu+(1-t)u_0)-f(u)\right)\\
&\le \Delta_{f,D,u_0}^a(t)(\lambda a+\vert f_\lambda(u)\vert)\notag\\
&\le \eps\max\{\lambda a,1\}(1+\vert f_\lambda(u)\vert).\notag
\end{align*}
It follows that $\limsup_{t\to 1}\Delta_{f_\lambda,D,u_0}^{1}(t)\le  \eps\max\{\lambda a,1\}$. The proof of~\ref{extp} is complete by letting $\eps\to 0$.
\end{proof}
The stability for the operations sum and product (pointwise) of ru-usc functions is specified below.
\begin{proposition}\label{sumproduct} Let $f,g:X\to ]-\infty,\infty]$ be two ru-usc functions in $D\subset\dom f\cap\dom g$ relative to $u_0\in D$. 
\begin{hyp1}
\item\label{sum} If $\inf_D f\ssup -\infty$ and $\inf_D g\ssup -\infty$ then $f+g$ is ru-usc in $D$ relative to $u_0$;\\
\item\label{product} If $\inf_D f\ssup 0$ and $\inf_D g\ssup 0$ then $fg$ is ru-usc in $D$ relative to $u_0$.
\end{hyp1}
\end{proposition}
\begin{proof} {\em Proof of~\ref{sum}}. Assume first that $f\ge 0$ and $g\ge 0$ on $D$. Fix $\eps\ssup0$. There exists $t_\eps\in ]0,1[$ such that $\sup_{t\in ]t_\eps,1[} \Delta_{f,D,u_0}^a(t)\sinf\eps$ (resp. $\sup_{t\in ]t_\eps,1[} \Delta_{g,D,u_0}^b(t)\sinf\eps$) since $f$ (resp. $g$) is ru-usc in $D$ relative to $u_0$. Fix $t\in ]t_\eps,1[$ and $u\in D$. We have
\begin{align} 
f(tu+(1-t)u_0)\le \eps\left(a+f(u)\right)+f(u)\label{fruusc}\\
g(tu+(1-t)u_0)\le \eps\left(b+g(u)\right)+g(u)\label{gruusc},
\end{align}
 so by adding~\eqref{fruusc} with~\eqref{gruusc} we obtain
\begin{align*}
f(tu+(1-t)u_0)+g(tu+(1-t)u_0)\le \delta(\eps)(1+f(u)+g(u))+f(u)+g(u)
\end{align*}
 where $\delta(\eps):=\eps\max\{a+b,1\}$ and satisfies $\lim_{\eps\to 0}\delta(\eps)=0$. Therefore \[\limsup_{t\to 1}\Delta^1_{f+g,D,u_0}(t)\le \delta(\eps)\] which gives the result by letting $\eps\to 0$.

We now remove the restrictions made on $f$ and $g$ and we set $m:=\inf_D f+\inf_D g$. We denote by $f^+:=f-\inf_D f$ and $g^+:=g-\inf_Dg$. By Lemma~\ref{translation}~\ref{transla} the functions $f^+$ and $g^+$ are ru-usc in $D$ relative to $u_0$, so by applying the first part of the Proof of~\ref{sum} the function $f^++g^+$ is ru-usc in $D$ relative to $u_0$. But $f^++g^++m=f+g$, so again by applying Lemma~\ref{translation} we find that $f+g$ is ru-usc in $D$ relative to $u_0$. 
 
\noindent{\em Proof of~\ref{product}}. By taking the product of~\eqref{fruusc} with~\eqref{gruusc} we obtain
\begin{align}\label{eqproduct}
f(tu+(1-t)u_0)g(tu+(1-t)u_0)\le \delta(\eps)(1+f(u)+g(u)+f(u)g(u))+f(u)g(u)
\end{align}
where $\delta(\eps):=\max\left\{\eps,\eps^2\right\}\max\{ab,2a,2b,3\}$ and satisfies $\lim_{\eps\to 0}\delta(\eps)= 0$. But 
\begin{align}\label{eqproduct1}
f(u)+g(u)\le f(u)g(u)\left(\frac{1}{\inf_D f}+\frac{1}{\inf_Dg}\right).
\end{align}
From~\eqref{eqproduct} and~\eqref{eqproduct1} we deduce that
\begin{align*}
f(tu+(1-t)u_0)g(tu+(1-t)u_0)\le &\delta(\eps)\left(1+\frac{1}{\inf_D f}+\frac{1}{\inf_Dg}\right)(1+f(u)g(u))\\&+f(u)g(u),
\end{align*}
so $\limsup_{t\to 1}\Delta^1_{fg,D,u_0}(t)\le \delta(\eps)\left(1+\frac{1}{\inf_D f}+\frac{1}{\inf_Dg}\right)$ which gives the result by letting $\eps\to 0$.

\end{proof}

\begin{remark} The Proposition~\ref{sumproduct}~\ref{product} answers to the question whether class of ru-usc functions contains more than convex functions, indeed it is sufficient to consider two finite positive convex functions such that their pointwise product is not convex.
\end{remark}

\begin{example}\label{remholder} Assume that $X$ is a normed space. Let $g:X\to \RR$ be a function satisfying:
\begin{hypA}
\item\label{A1} there exist $\alpha\ge 0$ and a function $\delta:[0,1]\to \RR$ satisfying $\limsup_{t\to 1}\delta(t)\le 0$ such that for every $t\in ]0,1[$ and every $u\in X$
\begin{align*}
\vert g(tu+(1-t)u_0)-g(u)\vert\le \delta(t)(1+\Vert u\Vert^\alpha+\Vert u_0\Vert^\alpha);
\end{align*}
\item\label{A2} there exist $c\ssup 0$, $c^\prime\ge 0$, and $\beta\ge \alpha$ such that for every $u\in X$
\begin{align*}
c\Vert u\Vert^\beta-c^\prime \le g(u);
\end{align*}
\end{hypA}
then $g$ is ru-usc in $D$ relative to $u_0$. Indeed, by~\ref{A2} we have $\Vert u\Vert^\alpha\le \Vert u\Vert^\beta+1\le\frac{1}{c}\left(g(u)+c^\prime\right)+1$. Set $a:=c^\prime+c(2+\Vert u_0\Vert^\alpha)$ then it is easy to deduce from~\ref{A1} that for every $u\in D$ and every $t\in [0,1[$
\begin{align*}
g(tu+(1-t)u_0)-g(u)\le \frac{\delta(t)}{c}\left(a+g(u)\right).
\end{align*}
Therefore $\Delta^{a}_{g,D,u_0}(t)\le \frac{\delta(t)}{c}$ which shows,  by letting $t\to 1$, that $g$ is ru-usc in $D$ relative to $u_0$. Note also that $\inf_X g\ssup -\infty$.  
\end{example}

\begin{corollary}\label{rh} Assume that $X$ is a normed vector space. Let $f:X\to ]-\infty,\infty]$ be a ru-usc function in $D\subset \dom f$ relative to $u_0\in D$. If $\inf_D f\ssup-\infty$ and if $g:X\to ]-\infty,\infty]$ satisfies~\ref{A1} and~\ref{A2} then $f+g$ is ru-usc in $D$ relative to $u_0$.
\end{corollary}
\begin{remark} Corollary~\ref{rh} can be seen as stability of ru-usc functions with respect to a type of ``radial H\"older" perturbation.
\end{remark}

The following result is an alternative to the Proposition~\ref{sumproduct}~\ref{sum}.

\begin{proposition}\label{boundedruusc+ruusc} Let $f,g:X\to ]-\infty,\infty]$ be two ru-usc functions in $D\subset\dom f\cap\dom g$ relative to $u_0\in D$. If $g$ is bounded on $D$, i.e., $\sup_{u\in D}\vert g(u)\vert\sinf\infty$ then $f+g$ is ru-usc in $D$ relative to $u_0$.
\end{proposition}
\begin{proof} Fix $\eps\ssup0$. There exists $t_\eps\in ]0,1[$ such that $\sup_{t\in ]t_\eps,1[} \Delta_{f,D,u_0}^a(t)\sinf\eps$ (resp. $\sup_{t\in ]t_\eps,1[} \Delta_{g,D,u_0}^b(t)\sinf\eps$) since $f$ (resp. $g$) is ru-usc in $D$ relative to $u_0$. Fix $t\in ]t_\eps,1[$ and $u\in D$. We have
\begin{align} 
f(tu+(1-t)u_0)\le \eps\left(a+\lvert f(u)\rvert\right)+f(u)\label{fruusc1}\\
g(tu+(1-t)u_0)\le \eps\left(b+\lvert g(u)\rvert\right)+g(u)\label{gruusc1},
\end{align}
so by adding~\eqref{fruusc1} with~\eqref{gruusc1} we obtain
\begin{align*}
&f(tu+(1-t)u_0)+g(tu+(1-t)u_0)\\
&\le \delta(\eps)(1+\vert f(u)\vert+\vert g(u)\vert)+f(u)+g(u)\\
&\le \delta(\eps)(1+\vert f(u)+g(u)\vert+2\sup_{u\in D}\vert g(u)\vert)+f(u)+g(u)\\
&\le\delta(\eps)\max\left\{1,2\sup_{u\in D}\vert g(u)\vert\right\}\left(1+\vert f(u)+g(u)\vert\right) +f(u)+g(u)
\end{align*}
where $\delta(\eps):=\eps\max\{a+b,1\}$, and satisfies $\lim_{\eps\to 0}\delta(\eps)=0$. Then 
\[
\limsup_{t\to1}\Delta^1_{f+g,D,u_0}(t)\le \delta(\eps)\max\left\{1,2\sup_{u\in D}\vert g(u)\vert\right\}
\] 
which gives the result by letting $\eps\to 0$.
\end{proof}
\begin{example}\label{uniformc} Assume that $X$ is a normed vector space. Let $D\subset X$ be a compact and strongly star-shaped set relative to $u_0\in D$. Let $g:X\to ]-\infty,\infty]$ be a function which is continuous and finite on $D$. Then $g$ is ru-usc in $D$ relative to $u_0$. Indeed, let $\omega:[0,\infty[\to\RR$ be defined by
\begin{align*}
\omega(\delta)=\sup\{\vert g(u)-g(v)\vert:u,v\in {D}\mbox{ and }\Vert u-v\Vert\sinf\delta\}.
\end{align*}
We have $\lim_{\delta\to 0}\omega(\delta)=0$ since $g$ is continuous and $D$ is compact. We have for every $u\in D$ and every $t\in [0,1[$
\begin{align*}
g(tu+(1-t)u_0)-g(u)\le \omega\left((1-t)\sup_{u\in{D}}\Vert u-u_0\Vert\right).
\end{align*}
Therefore we have $\Delta_{g,D,u_0}^1(t)\le \omega\left((1-t)\sup_{u\in{D}}\Vert u-u_0\Vert\right)$ for all $t\in [0,1[$. Passing to the limit $t\to 1$ we obtain that $g$ is ru-usc in $D$ relative to $u_0$.
\end{example}
By using Example~\ref{uniformc} and Proposition~\ref{boundedruusc+ruusc} we establish, in the following result, that the class of ru-usc functions are stable with respect to a continuous perturbation when $D$ is strongly star-shaped and compact.
\begin{corollary} Assume that $X$ is a normed vector space and $D\subset X$ is a compact and strongly star-shaped set relative to $u_0\in D$. Let $f:X\to ]-\infty,\infty]$ be a ru-usc function in $D\subset\dom f$ relative to $u_0$. If $g:X\to \RR$ is a continuous function then $f+g$ is ru-usc in $D$ relative to $u_0$.
\end{corollary}
\subsection{Inf-convolution of ru-usc functions}
For two functions $f,g:X\to ]-\infty,\infty]$ their inf-convolution is the function $f\triangledown g:X\to [-\infty,\infty]$ defined by
\[
\left(f\triangledown g\right)(u):=\inf_{v\in X}\left\{f(u-v)+g(v)\right\}.
\]
The following result establishes the conditions to keep the ru-usc property by the inf-convolution operation.
\begin{proposition}\label{inf-convol} Let $f:X\to ]-\infty,\infty]$ be a ru-usc function relative to $u_0\in \dom f$. Let $g:X\to ]-\infty,\infty]$ be a ru-usc function relative to $0\in \dom g$. Then $f\triangledown g$ is ru-usc relative to $u_0$ if one of the following conditions holds:
\begin{hyp1}
\item\label{minconv} $\inf_{X} f\ssup -\infty$ and $\inf_{X} g\ssup -\infty$;\\

\item\label{convboun} $\displaystyle\sup_{v\in\dom g}\vert g(v)\vert\sinf\infty$.
\end{hyp1}
\end{proposition}
\begin{proof} {\em Proof of~\ref{minconv}.} Assume first that $f\ge 0$ and $g\ge 0$. Fix $u\in \dom f\triangledown g$. Choose $\{v_n\}_n\subset \dom g$ such that $f(u-v_n)+g(v_n)\to \left(f\triangledown g\right)(u)$ as $n\to\infty$. Fix $\eps\ssup0$. There exists $t_\eps\in ]0,1[$ such that $\sup_{t\in ]t_\eps,1[} \Delta_{f,D,u_0}^a(t)\sinf\eps$ (resp. $\sup_{t\in ]t_\eps,1[} \Delta_{g,D,u_0}^b(t)\sinf\eps$) since $f$ (resp. $g$) is ru-usc in $D$ relative to $u_0$. Fix $t\in ]t_\eps,1[$. We have for every $n\in\NN$ 
\begin{align}\label{eqinf0}
\left(f\triangledown g\right)(tu+(1-t)u_0)&\le f(t(u-v_n)+(1-t)u_0)+g(tv_n)\\
&\le\eps(a+f(u+v_n))+\eps(b+g(v_n))\notag\\
&+f(u-v_n)+g(v_n)\notag\\
&\le\delta(\eps)(1+f(u-v_n)+g(v_n))\notag\\
&+f(u-v_n)+g(v_n).\notag
\end{align}
where $\delta(\eps):= \eps\max\{1,a+b\}$ and satisfies $\limsup_{\eps\to 0}\delta(\eps)\le 0$. Letting $n\to\infty$ in~\eqref{eqinf0} we obtain $\limsup_{t\to 1}\Delta_{f\triangledown g,u_0}^{1}(t)\le \delta(\eps)$ which gives, by letting $\eps\to 0$, that $f\triangledown g$ is ru-usc relative to $u_0$. 

We remove the restrictions on $f$ and $g$. Set $f^+:= f-\inf_X f$ and $g^+:= g-\inf_X f$. By Lemma~\ref{translation} the function $f^+$ (resp. $g^+$) is ru-usc relative to $u_0$ (resp. relative to $0$), so we apply the first part of the proof to have $f^+\triangledown g^+$ is ru-usc relative to $u_0$. But $f^+\triangledown g^++(\inf_X f+\inf_X g)=f\triangledown g$, so again by Lemma~\ref{translation} we deduce that $f\triangledown g$ is ru-usc relative to $u_0$.

\noindent{\em Proof of~\ref{convboun}.} Choose $\{v_n\}_n\subset \dom g$ such that $f(u-v_n)+g(v_n)\to \left(f\triangledown g\right)(u)$ as $n\to\infty$. Fix $t\in ]t_\eps,1[$. Then for every $n\in\NN$
\begin{align}\label{eqinf}
\left(f\triangledown g\right)(tu+(1-t)u_0)&\le f(t(u-v_n)+(1-t)u_0)+g(tv_n)\\
&\le\eps(a+\vert f(u-v_n)\vert)+\eps(b+\vert g(v_n)\vert )\notag\\
&+f(u-v_n)+g(v_n)\notag\\
&\le\delta(\eps)(1+\vert f(u-v_n)\vert+\vert g(v_n)\vert)\notag\\
&+f(u-v_n)+g(v_n)\notag\\
&\le\delta(\eps)\left(1+\vert f(u-v_n)+g(v_n)\vert+2\sup_{v\in\dom g}\vert g(v)\vert\right)\notag\\
&+f(u-v_n)+g(v_n)\notag
\end{align}
where $\delta(\eps):= \eps\max\{1,a+b\}$ and satisfies $\lim_{\eps\to 0}\delta(\eps)=0$. Letting $n\to\infty$ in~\eqref{eqinf} we obtain $\limsup_{t\to1}\Delta_{f\triangledown g,u_0}^{1}(t)\le \delta(\eps)\left(1+2\sup_{v\in\dom g}\vert g(v)\vert\right)$ which gives, by letting $\eps\to 0$, that $f\triangledown g$ is ru-usc relative to $u_0$. 
\end{proof}

If $g=\chi_C$ with $C\subset X$ a strongly star-shaped set relative to $0\in C$, then $g$ is ru-usc relative to $0$ since Remark~\ref{indicator-ruusc}. By noticing that $\sup_{v\in\dom g}\vert g(v)\vert=0\sinf\infty$, we may apply Proposition~\ref{inf-convol}~\ref{convboun} to obtain that the function $X\ni u\mapsto \left(f\triangledown \chi_C\right)(u)=\inf_{v\in C}f(u-v)$ is ru-usc when $f$ is ru-usc relative to $u_0\in\dom f$.

\begin{corollary} Let $f:X\to ]-\infty,\infty]$ be a ru-usc function relative to $u_0\in \dom f$. Let $C\subset X$ be a strongly star-shaped set relative to $0\in C$. Then $f\triangledown\chi_C$ is ru-usc relative to $u_0$.
\end{corollary}

\section{Application to the relaxation with constraints}\label{4}
Let $d,m\ge 1$ be two integers and $p\in]1,\infty[$. Let $\Omega\subset\RR^d$ be a bounded open set with Lipschitz boundary. We denote by $\MM^{m\times d}$ the space of $m$ rows and $d$ columns matrices.
\subsection{Star-shaped subsets in $W^{1,p}(\Omega;\RR^m)$}
We consider the class of subsets $S\subset\MM^{m\times d}$ satisfying 
\begin{hypHH}
\item\label{H3} $0\in S$;
\item\label{H4} for every sequence $\{u_n\}_n\subset W^{1,p}(\Omega;\RR^m)$ such that $u_n\wto u$ in $W^{1,p}(\Omega;\RR^m)$, if for every $n\in\NN$ we have $\nabla u_n(\cdot)\in S$ a.e. in $\Omega$ then for every $t\in[0,1[$ it holds $t\nabla u(\cdot)\in S$ a.e. in $\Omega$.
\end{hypHH}
Define $\D\subset W^{1,p}(\Omega;\RR^m)$ by 
\begin{align}\label{Dconstraints}
\D:=\left\{u\in W^{1,p}(\Omega;\RR^m):\nabla u(x)\in S\mbox{ a.e. in }\Omega\right\}.
\end{align}
We say that $\D$ is {\em weakly sequentially strongly star-shaped relative to $0$ in $W^{1,p}(\Omega;\RR^m)$} when $W^{1,p}(\Omega;\RR^m)$ is endowed with the weak topology and $\D$ is sequentially strongly star-shaped relative to $0$. We denote by $\overseqw{D}$ the sequential weak closure of $\D$ in $W^{1,p}(\Omega;\RR^m)$.

The following result shows that the conditions~\ref{H3} and~\ref{H4} on $S$ give rise to sequentially strongly star-shaped sets of the form $\D$.
\begin{lemma}\label{D-starshaped} If $S$ satisfies~\ref{H3} and~\ref{H4} then $\D$ is weakly sequentially strongly star-shaped relative to $0$. 
\end{lemma}
\begin{proof} For every $t\in[0,1[$ and every $u\in\overseqw{\D}$ we have $tu\in\D$ since~\ref{H4}.
\end{proof}

\begin{example} It is not difficult to see that if $S$ is convex with $0\in \inte(S)$ then $S$ satisfies~\ref{H3} and~\ref{H4}. Indeed, using Mazur lemma we have $\nabla u(\cdot)\in \overline{S}$ a.e. in $\Omega$. Then for every $t\in [0,1[$ it holds $t\nabla u(\cdot)\in S$ a.e. in $\Omega$ since $\overline{S}$ is convex and $0\in\inte(S)$.
\end{example}
We give an example of a family of nonconvex sets satisfying~\ref{H3} and~\ref{H4}.
\begin{example} Assume that $m=d=2$. For each $\eps\ssup 0$ consider the set 
\[
S_\eps:=\left\{\xi\in\MM^{2\times 2}:\eps+\det(\xi)\ssup \tr(\xi)^2\right\}.
\]
We have for every $\eps\ssup 0$
\begin{hyp1}
\item\label{st-1} $0\in S_\eps$ ( i.e.~\ref{H3} is satisfied);
\item\label{st0} for every $\xi\in\overline{S_\eps}$ we have $t\xi\in S_\eps$ for all $t\in [0,1[$;
\item\label{st1} $S_\eps$ is not convex;
\item\label{st2} $S_\eps$ is not bounded;
\item\label{st3} $S_\eps$ is rank-one convex.
\end{hyp1}
Fix $\eps\ssup 0$. The set $S_\eps$ satisfies~\ref{H4} for all $p\ssup 2$. Indeed, let $\{u_n\}_n\subset W^{1,p}(\Omega;\RR^m)$ be such that $u_n\wto u$ in $W^{1,p}(\Omega;\RR^m)$ and for every $n\in\NN$ it holds $\nabla u_n(\cdot)\in S_\eps$ a.e. in $\Omega$. By a classical result (see~\cite[Theorem 8.20, p. 395]{dacorogna08}) we have 
\begin{align}\label{det}
\det(\nabla u_n(\cdot))\wto\det(\nabla u(\cdot)) \mbox{ in }L^{p\over 2}(\Omega;\RR^m).
\end{align}
The function $\tr(\cdot)^2$ is convex and continuous, so for every Borel set $A\subset\Omega$ we have
\begin{align}\label{tr}
\liminf_{n\to \infty}\int_A \tr(\nabla u_n(x))^2dx\ge \int_A \tr(\nabla u(x))^2dx.
\end{align}
Using~\eqref{det} and~\eqref{tr} we find for almost all $x\in\Omega$ and $\rho\ssup 0$
\begin{align*}
\eps+\frac{1}{\vert B_\rho(x)\vert}\int_{B_\rho(x)} \det(\nabla u(y))dy\ge \frac{1}{\vert B_\rho(x)\vert}\int_{B_\rho(x)} \tr(\nabla u(y))^2dy,
\end{align*}
then by passing to the limit $\rho\to 0$ we obtain $\nabla u(\cdot)\in \overline{S_\eps}$ a.e. in $\Omega$. By~\ref{st0} we have $t\nabla u(\cdot)\in S_\eps$ a.e. in $\Omega$ for all $t\in [0,1[$.

\noindent{The properties~\ref{st-1} and~\ref{st2} are immediate.}

\noindent{Proof of~\ref{st0}.} Let $t\in[0,1[$ and $\xi\in\overline{S_\eps}$. We have 
\begin{align*}
\det(t\xi)=t^2\det(\xi)\ge t^2(\tr(\xi)^2)-t^2\eps\ssup \tr(t\xi)^2-\eps.
\end{align*}
\\
\noindent{Proof of~\ref{st1}.} Consider
\begin{align*}
\xi:=
\left(
\begin{array}{ccc}
\sqrt{\frac{3}{2}\eps}  &   -\sqrt{\eps}   \\
 \sqrt{\eps} &        0
\end{array}
\right)\quad\mbox{ and }\quad
\zeta:=
\left(
\begin{array}{ccc}
0  &   \sqrt{\eps}   \\
 -\sqrt{\eps} &       \sqrt{\frac{3}{2}\eps}
\end{array}
\right)
\end{align*}
It is easy to see that $\xi,\zeta\in S_\eps$ and $\frac{1}{2}\xi+\frac{1}{2}\zeta\notin S_\eps$.
\\
\noindent{Proof of~\ref{st3}.} We have to show that $t\xi+(1-t)\zeta\in S_\eps$ for all $t\in ]0,1[$ whenever $\xi,\zeta\in S_\eps$ satisfy $\rk(\xi-\zeta)\le 1$. Property~\ref{st3} follows by using the fact that $\det(\cdot)$ is quasiaffine (see~\cite[Example 5.21(i), p. 179]{dacorogna08}) and $\tr(\cdot)^2$ is convex.
\end{example}

\subsection{Relaxation of multiple integrals with star-shaped constraints}
Let $L:\MM^{m\times d}\to [0,\infty[$ be a quasiconvex (in the sense of Morrey) integrand with $p$-polynomial growth, i.e., $L$ satisfies
\begin{hypH}
\item\label{H1}$\displaystyle L(\xi)=\inf\left\{\int_{]0,1[^d}L(\xi+\nabla u(x))dx:u\in W^{1,\infty}_0(]0,1[^d;\RR^m)\right\}\quad\mbox{ for all }\xi\in\MM^{m\times d}$;
\item\label{H2}There exist $c,C\ssup 0$ such that $c\vert \xi\vert^p\le L(\xi)\le C(1+\vert\xi\vert^p)$ for all $\xi\in\MM^{m\times d}$.
\end{hypH}
Define the integral functional $F:W^{1,p}(\Omega;\RR^m)\to [0,\infty]$ by 
\begin{align*}
F(u):=
\left\{
\begin{array}{cl}
\displaystyle \int_{\Omega} L(\nabla u(x))dx & \mbox{ if }\nabla u(x)\in S\mbox{ a.e. in }\Omega   \\
  \infty&  \mbox{ otherwise.}
\end{array}
\right.
\end{align*}
Our goal here is to give a representation of the lower semicontinuous envelope $F$ in $W^{1,p}(\Omega;\RR^m)$ endowed with the strong topology of $L^{p}(\Omega;\RR^m)$:

\begin{align*}
\overline{F}(u):=\inf\left\{\liminf_{n\to \infty} F(u_n):W^{1,p}(\Omega;\RR^m)\ni u_n\stackrel{L^p}{\to} u\right\}.
\end{align*} 

If $\L:W^{1,p}(\Omega;\RR^m)\to [0,\infty]$ is defined by 
\begin{align*}
\L(u):=\int_\Omega L(\nabla u(x))dx.
\end{align*}
then for every $u\in W^{1,p}(\Omega;\RR^m)$
\begin{align*}
{F}(u)=\L(u)+\chi_{\D}(u),
\end{align*}
where $\D$ is given by~\eqref{Dconstraints}.
\begin{lemma}\label{ru-usc-qc} If $L$ satisfies~\ref{H1} and~\ref{H2} then $\L$ is ru-usc in $\D$ relative to $0$.
\end{lemma}
\begin{proof} Since $L$ satisfies~\ref{H1} and~\ref{H2} we have for some $C^\prime\ssup 0$
\begin{align}\label{lipquasi}
\vert L(\xi)-L(\zeta)\vert\le C^\prime\vert \xi-\zeta\vert(1+\vert\xi\vert^{p-1}+\vert\zeta\vert^{p-1})
\end{align}
for all $\xi,\zeta\in\MM^{m\times d}$ (see for instance~\cite[Proposition 2.32, p. 51]{dacorogna08}). For every $u\in\D$ and every $t\in [0,1[$
\begin{align*}
\vert\L(tu)-\L(u)\vert&\le\int_\Omega 2C^\prime(1-t)\vert\nabla u(x)\vert(1+\vert\nabla u(x)\vert^{p-1})dx\\
&\le 4C^\prime(1-t)\int_\Omega 1+\vert\nabla u(x)\vert^{p}dx\\
&\le 4C^\prime(1-t)\int_\Omega 1+\frac1c L(\nabla u(x))dx\\
&\le 4C^\prime\max\left\{1,\frac1c\right\}(1-t)(\vert\Omega\vert+\L(u)).
\end{align*}
Then $\Delta_{\L,\D,0}^{\vert\Omega\vert}(t)\le 4C^\prime\max\left\{1,\frac1c\right\}(1-t)$, the proof is complete by letting $t\to 1$. 
\end{proof}

\begin{theorem}\label{quasiconstraints} If $S$ satisfies~\ref{H3} and~\ref{H4}, and if $L$ satisfies~\ref{H1} and~\ref{H2}, then for every $u\in W^{1,p}(\Omega;\RR^m)$
\begin{align}\label{formulaintegral}
\overline{F}(u)=J(u)+\chi_{\overseqw{\D}}(u).
\end{align}
\end{theorem}
\begin{proof} The sequential relaxation of a functional $\Phi:W^{1,p}(\Omega;\RR^m)\to [0,\infty]$ with respect to the weak convergence in $W^{1,p}(\Omega;\RR^m)$ is given by
\begin{align*}
\overseqw{\Phi}(u):=\inf\left\{\liminf_{n\to \infty} \Phi(u_n):W^{1,p}(\Omega;\RR^m)\ni u_n\wto u\right\}.
\end{align*} 
Note that $\overline{F}=\overseqw{F}$ since the coercivity condition~\ref{H2}. So, it suffices to show that 
\begin{align*}
\overseqw{F}(u)=J(u)+\chi_{\overseqw{\D}}(u).
\end{align*}

We have $\overseqw \L=\L$ since~\ref{H1} and~\ref{H2} (see for instance~\cite{acerbi-fusco84}). Then it holds $\overseqw{\L+\chi_\D}=\L$ on $\D$, indeed we have
\begin{align*}
\L=\overseqw{\L}\le \overseqw{\L+\chi_\D}\le \L+\chi_\D.
\end{align*}
Using Lemma~\ref{ru-usc-qc} and Lemma~\ref{D-starshaped} we obtain by Theorem~\ref{GENERAL-FORMULA} 
\begin{align*}
\overseqw{F}=\widehat{\L}_0+\chi_{\overseqw\D}\quad\mbox{ with }\quad\widehat{\L}_0(u)=\lim_{t\uparrow 1}\L(tu).
\end{align*}
It remains to prove that $\widehat{\L}_0=\L$ on $\overseqw\D\ssetminus\D$ since $\overseqw{F}=\overseqw{\L+\chi_\D}=\L$ on $\D$. Fix $u\in \overseqw\D\ssetminus\D$. Using the polynomial growth~\ref{H2} we have for almost all $x\in\Omega$
\[
\sup_{t\in [0,1[}L(t\nabla u(x))\le C(1+\Vert \nabla u(x)\Vert^p).
\]
The integrand $L$ is continuous since~\eqref{lipquasi}. We may apply the dominated convergence theorem to obtain
\begin{align*}
\widehat{\L}_0(u)=\lim_{t\uparrow 1}\L(tu)=\int_\Omega\lim_{t\uparrow 1} L(t\nabla u(x))dx=\L(u).
\end{align*}
The proof is complete.
\end{proof}
\begin{remark} The equality~\eqref{formulaintegral}, which can be rewritten as 
\begin{align}\label{formulaintegral2}
\overline{J+\chi_D}=J+\chi_{\overseqw{\D}},
\end{align}
looks natural since $J$ is of ``$p$-polynomial growth" (and lsc with respect to the strong topology in $L^p(\Omega;\RR^m)$) and since the star-shaped property (of the constraints) can be seen as a kind of ``regularity" on the constraints. An interesting further extension is to study whether similar equality holds when we replace the lsc envelope by a $\Gamma$-limit procedure and $J$ by a sequence of funtionals $\{J_n\}_n$ (for an interesting discussion about constrained problems see~\cite[p. 499]{degiorgi79}).
\end{remark}

\end{document}